%% file: paper.tex
\documentclass[10pt, a4paper, fleqn, final]{article}
\usepackage[utf8x]{inputenc}
\usepackage[T1]{fontenc}
\usepackage[english]{babel}
\usepackage{amsmath, amsthm, amssymb, mathtools}
\usepackage{libertine}
\usepackage{graphicx}
\usepackage{url}

\newtheorem{theo}{Theorem}
\newtheorem{conj}[theo]{Conjecture}
\newtheorem{coro}[theo]{Corollary}
\newtheorem{lemm}[theo]{Lemma}
\newcommand{\eqspace}{\ensuremath{\mathrel{\phantom{=}}}}
\newcommand{\cdotcup}{\mathrel{\mathaccent\cdot\cup}}

\title{The Merrifield-Simmons conjecture holds for bipartite graphs}
\author{Martin Trinks\thanks{trinks@hs-mittweida.de, Hochschule Mittweida, University of Applied Sciences, Faculty Mathematics / Sciences / Computer Science, Technikumplatz 17, 09648 Mittweida, Germany}}
\date{\today}

\begin{document}

\maketitle

{\center\setlength{\parindent}{0pt}
\begin{minipage}[c]{5.7cm}
The author receives a grant from the European Social Fund (ESF) of the European Union (EU).
\end{minipage}
\hspace{0.4cm}
\begin{minipage}[c]{5.7cm}
\centering
\includegraphics[width=3cm]{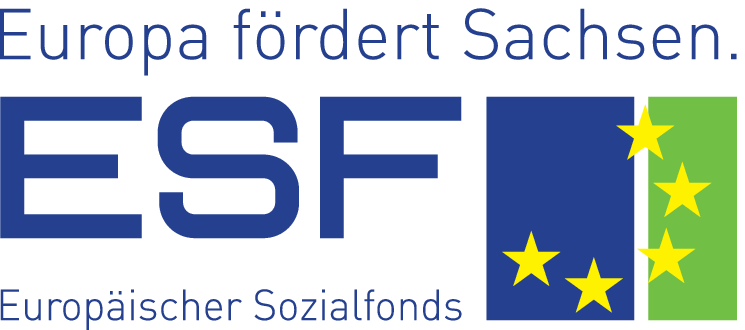}
\hspace{0.25cm}
\includegraphics[width=2cm]{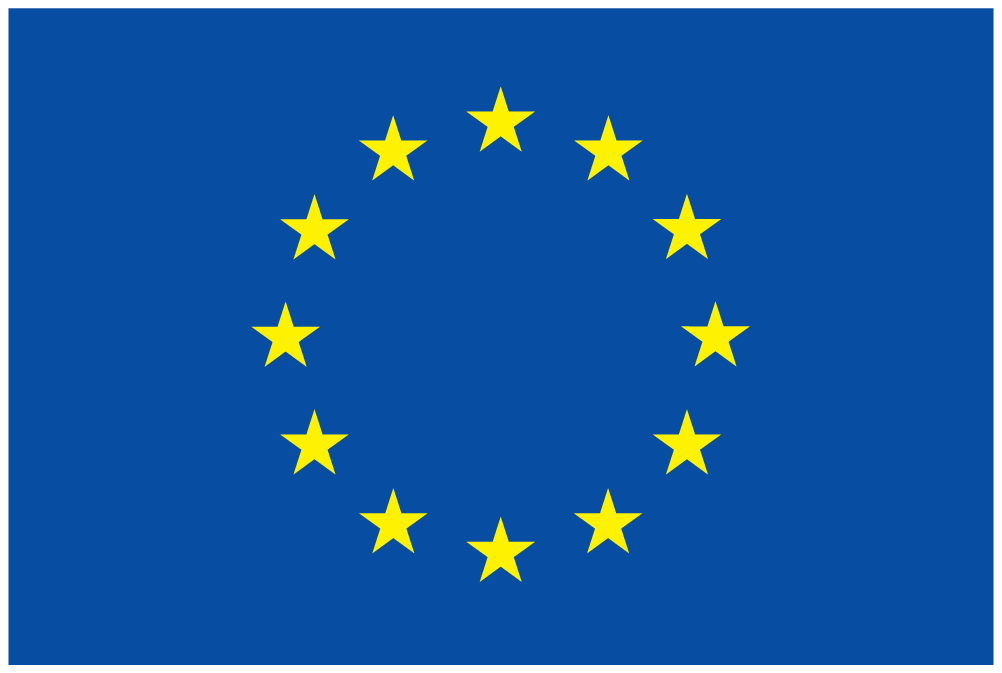}
\end{minipage}
}

\begin{abstract}
Let $G = (V, E)$ be a graph and $\sigma(G)$ the number of independent (vertex) sets in $G$. Then the Merrifield-Simmons conjecture states that the sign of the term $\sigma(G_{-u}) \cdot \sigma(G_{-v}) - \sigma(G) \cdot \sigma(G_{-u-v})$ only depends on the parity of the distance of the vertices $u, v \in V$ in $G$. We prove that the conjecture holds for bipartite graphs by considering a generalization of the term, where vertex subsets instead of vertices are deleted.
\end{abstract}

\input{introduction}
\input{gmsc}
\input{main_theorem}

\end{document}

%% file: introduction.tex
\section{Introduction}

Let $G = (V, E)$ be a graph and $u$, $v \in V$ two vertices of $G$. We define the term $\Delta(G, u, v)$ as 
\begin{align}
\Delta(G, u, v) = \sigma(G_{-u}) \cdot \sigma(G_{-v}) - \sigma(G) \cdot \sigma(G_{-u-v}), \label{eq:delta_def}
\end{align}
where \(\sigma(G)\) is the number of independent (vertex) sets in $G$, i.e. the number of vertex subsets \(W \subseteq V\) such that no two vertices of \(W\) are adjacent, and where \(G_{-w}\) is the graph with the vertex \(w\) and its incident edges removed.

Merrifield and Simmons \cite{merrifield1989} note (without proof), that the sign of $\Delta(G, u, v)$  alternates with respect to $d_G(u, v)$, the distance of the vertices $u$ and $v$ in the graph $G$. Although they probably consider only graphs representing molecular structures (bipartite graphs \cite{gutman1990}), this property is also of interest for arbitrary graphs. Hence it is generalized to determine for which graphs and vertices the following conjecture holds:

\begin{conj}[Merrifield-Simmons conjecture]
\label{conj:msc}
Let $G = (V, E)$ be a graph and $u, v \in V$ two vertices of $G$. Then
\begin{align}
\Delta(G, u, v) &= \sigma(G_{-u}) \cdot \sigma(G_{-v}) - \sigma(G) \cdot \sigma(G_{-u-v}) \notag \\
&\begin{cases}
< 0 & \text{if } d_G(u, v) \text{ is even,} \\
> 0 & \text{if } d_G(u, v) \text{ is odd.}
\end{cases} \label{eq:conj_msc}
\end{align}
\end{conj}

In fact the authors note this property for a related term, the propagator $\gamma_{u, v}(G)$, such that for nonadjacent vertices $u$ and $v$ the above conjecture equals the original one and otherwise is a necessary condition \cite{gutman1990}.

Gutman \cite{gutman1990} mentions examples of non-bipartite graphs and nonadjacent vertices, showing that the Merrifield-Simmons conjecture (Conjecture \ref{conj:msc}) is not true in general, but proves that it holds in arbitrary graphs for adjacent vertices and restates the conjecture, that it is valid in bipartite graphs for arbitrary vertices. The same author verifies the conjecture for trees by applying a more general theorem \cite{gutman1991}. Li \cite{li1996} gives a proof in case of cycles and some small other graph classes. That the conjecture holds for bipartite unicyclic graphs was shown by Wang, Li, and Gutman \cite{wang2001}.

For a graph \(G = (V, E)\) with a vertex \(v \in V\) and a vertex subset \(W \subseteq V\) we use the following notation: By \(G_{-W}\) we denote the graph \(G\) where all vertices \(v \in W\) are deleted, that means these vertices and their incident edges are removed. The open neighborhood of $W$ is denoted by \(N_{G}(W)\), i.e. the set of all vertices adjacent to a vertex \(v \in W\). If \(W = \{v\}\) then we write \(G_{-v}\) and \(N_{G}(v)\) instead of \(G_{-\{v\}}\) and \(N_{G}(\{v\})\), respectively. \(G^1 \cdotcup G^2\) is the disjoint union of the graphs \(G^1\) and \(G^2\), i.e. the union of disjoint copies of both graphs. For all other notation we refer to \cite{diestel2005}.

For the number of independent sets \(\sigma(G)\) we use the following basic properties \cite{merrifield1981}: The number \(\sigma(G)\) is multiplicative in components, that means if \(G\) is the disjoint union of the graphs \(G_1\) and \(G_2\), then
\begin{align}
\sigma(G) &= \sigma(G^1 \cdotcup G^2) = \sigma(G^1) \cdot \sigma(G^2). \label{eq:sigma_prop_4}
\end{align}
Furthermore, it satisfies for all vertices $v \in V$ the recurrence relation
\begin{align}
\sigma(G) &= \sigma(G_{-v}) + \sigma(G_{-v-N_G(v)}). \label{eq:sigma_prop_1}
\end{align}
Consequently, by removing a vertex or a vertex subset from a graph, the number of independent sets of a graph decreases. Each graph, particularly the empty graph, has at least one independent set, namely the empty set.

We generalize the recurrence relation given in Equation \eqref{eq:sigma_prop_1} as follows:
\begin{theo}
\label{theo:sigma_prop_5}
Let \(G = (V, E)\) be a graph and \(U \subseteq V\) a vertex subset of \(G\). Then
\begin{align}
& \sigma(G) = \eqspace \sum_{\mathclap{\substack{W \subseteq U \\ W \text{ is independent}}}}{\sigma(G_{-U-N_{G}(W)})}. \label{eq:theo_sigma_prop_5}
\end{align}
\end{theo}
\begin{proof}
We prove the statement by induction on the number of vertices in \(U \subseteq V\). As basic step we consider the case \(|U| = 1\), that means \(U = \{u\}\). By Equation \eqref{eq:sigma_prop_1} we have
\begin{align*}
\sigma(G) &= \sigma(G_{-u}) + \sigma(G_{-u-N_{G}(u)}), \\
&= \eqspace \sum_{\mathclap{\substack{W \subseteq U = \{u\} \\ W \text{ is independent}}}}{\sigma(G_{-U-N_{G}(W)})}.
\end{align*}
As induction hypothesis we assume that the theorem holds for all sets \(U\) with \(|U| \leq n\). Then for a vertex subset \(U = U' \cup \{u\}\) with \(|U'| = n\) and \(u \notin U'\) we get by Equation \eqref{eq:sigma_prop_1} and by application of the induction hypothesis:
\begin{align*}
\sigma(G) &= \sigma(G_{-u}) + \sigma(G_{-u-N_{G}(u)}) \\
&= \eqspace \sum_{\mathclap{\substack{W' \subseteq U' \\ W' \text{ is independent}}}}{\sigma(G_{-u-U'-N_{G}(W')})} + \eqspace \sum_{\mathclap{\substack{W' \subseteq U' \setminus N_{G}(u) \\ W' \text{ is independent}}}}{\sigma(G_{-u-N_{G}(u)-U'-N_{G}(W')})} \\
&= \eqspace \sum_{\mathclap{\substack{W \subseteq U \\ W \text{ is independent} \\ u \notin W}}}{\sigma(G_{-U-N_{G}(W)})} + \eqspace \sum_{\mathclap{\substack{W \subseteq U \\ W \text{ is independent} \\ u \in W}}}{\sigma(G_{-U-N_{G}(W)})} \\  
&= \eqspace \sum_{\mathclap{\substack{W \subseteq U \\ W \text{ is independent}}}}{\sigma(G_{-U-N_{G}(W)})},
\end{align*}
and thus it holds also for \(|U| = n+1\).
\end{proof}

%% file: gmsc.tex
\section{A generalization for vertex subsets}

We introduce a generalization of the term $\Delta(G, u, v)$, where instead of vertices we delete vertex subsets. For a graph $G = (V, E)$ and two vertex subsets $A, B \subseteq V$ of $G$, we define the term $\Delta(G, A, B)$ as
\begin{align}
\Delta(G, A, B) = \sigma(G_{-A}) \cdot \sigma(G_{-B}) - \sigma(G) \cdot \sigma(G_{-A-B}).
\end{align}
Obviously, $\Delta(G, A, B)$ is symmetric with respect to the vertex sets $A$ and $B$, i.e.
\begin{align}
\Delta(G, A, B) = \Delta(G, B, A). \label{eq:gmsc_sym}
\end{align}

We investigate the dependence of the sign of $\Delta(G, A, B)$ on the length of all $A$-$B$-paths. 

A path $P = (v_1, \ldots, v_k)$ of $G$ is an $A$-$B$-path, if $V(P) \cap A = \{v_1\}$ and $V(P) \cap B = \{v_k\}$, where $V(P)$ is the set of vertices of $P$ \cite[Section 1.3]{diestel2005}. By $P(G, A, B)$ we denote the set of all $A$-$B$-paths in $G$. The length of an $A$-$B$-path $P$ is the number of edges in $P$, that means $|V(P)| - 1$. We say $P(G, A, B)$ is even (odd) if the length of each path $P \in P(G, A, B)$ is even (odd) and $P(G, A, B)$ is infinite, if there is no $A$-$B$-path in $G$ (the length of each $P \in P(G, A, B)$ is infinite).

Hence we want to generalize the Merrifield-Simmons conjecture by asking for which graphs and which vertex sets the following conjecture holds:
\begin{conj}[generalized Merrifield-Simmons conjecture]
\label{conj:gmsc}
Let $G=(V, E)$ be a graph and $A, B \subseteq V$ two vertex subsets of $G$. Then
\begin{align}
\Delta(G, A, B) &= \sigma(G_{-A}) \cdot \sigma(G_{-B}) - \sigma(G) \cdot \sigma(G_{-A-B}) \notag \\
&\begin{cases}
< 0 & \text{if } P(G, A, B) \text{ is even,} \\
= 0 & \text{if } P(G, A, B) \text{ is infinite,} \\
> 0 & \text{if } P(G, A, B) \text{ is odd.}
\end{cases} \label{eq:conj_gmsc}
\end{align}
\end{conj}

As for $\Delta(G, u, v)$ \cite{wang2001}, the sign of $\Delta(G, A, B)$ depends only on the connected components including at least one vertex of both vertex subsets $A$ and $B$.

\begin{theo}
\label{theo:delta_GAB}
Let $G = (V, E)$ be a graph and $A, B \subseteq V$ two vertex subsets of $G$. Then
\begin{align}
\Delta(G, A, B) = \sigma(G^{*}_{-A} \cdotcup G^{*}_{-B}) \cdot \Delta(G^{AB}, A, B),
\end{align}
where $G = G^{*} \cdotcup G^{AB}$ and $G^{AB}$ is the union of all connected components including at least one vertex of both vertex subsets $A$ and $B$.
\end{theo}

\begin{proof}
With $G = G^{*} \cdotcup G^{AB}$ and by Equation \eqref{eq:sigma_prop_4} (multiplicativity in components for the number of independent sets) we have
\begin{align*}
\Delta(G, A, B) &= \sigma(G_{-A}) \cdot \sigma(G_{-B}) - \sigma(G) \cdot \sigma(G_{-A-B}) \\
&= \sigma(G^{*}_{-A} \cdotcup G^{AB}_{-A}) \cdot \sigma(G^{*}_{-B} \cdotcup G^{AB}_{-B}) \\
& \eqspace - \sigma(G^{*} \cdotcup G^{AB}) \cdot \sigma(G^{*}_{-A-B} \cdotcup G^{AB}_{-A-B}) \\
&= \sigma(G^{*}_{-A} \cdotcup G^{*}_{-B}) \cdot \sigma(G^{AB}_{-A}) \cdot \sigma(G^{AB}_{-B}) \\
& \eqspace - \sigma(G^{*} \cdotcup G^{*}_{-A-B}) \cdot \sigma(G^{AB}) \cdot \sigma(G^{AB}_{-A-B}).
\end{align*}
By the definition of the graph $G^{*}$, the graph $G^{*}_{-A} \cdotcup G^{*}_{-B}$ is isomorphic to the graph $G^{*} \cdotcup G^{*}_{-A-B}$ and hence their number of independent vertex sets equals. It follows that
\begin{align*}
\Delta(G, A, B) &= \sigma(G^{*}_{-A} \cdotcup G^{*}_{-B}) \cdot [\sigma(G^{AB}_{-A}) \cdot \sigma(G^{AB}_{-B}) - \sigma(G^{AB}) \cdot \sigma(G^{AB}_{-A-B})] \\
&= \sigma(G^{*}_{-A} \cdotcup G^{*}_{-B}) \cdot \Delta(G^{AB}, A, B). \qedhere
\end{align*}
\end{proof}

From this we can derive a corollary for the case $P(G, A, B)$ is infinite.
\begin{coro}
\label{coro:delta_infty}
Let $G = (V, E)$ be a graph and $A, B \subseteq V$ two vertex subsets of $G$. If $P(G, A, B)$ is infinite, then
\begin{align}
\Delta(G, A, B) = 0.
\end{align}
\end{coro}
\begin{proof}
If $P(G, A, B)$ is infinite, then there is no $A$-$B$-path. Thus, no connected component of $G$ exists with at least one vertex from both vertex subsets $A$ and $B$. Consequently the graph $G^{AB}$ is the empty graph and $\Delta(G^{AB}, A, B) = 0$. By Theorem \ref{theo:delta_GAB} we get $\Delta(G, A, B) = 0$.
\end{proof}

%force UTF8 by äöü

%% file: main_theorem.tex
\section{Main theorem}

We use the following property of graphs.
\begin{lemm}
\label{lemm:gmsc_prop}
Let \(G = (V, E)\) be a graph, \(A, B \subseteq V\) two disjoint vertex subsets of \(G\) and \(W \subseteq A\) a subset of \(A\). If \(P(G, A, B)\) is even (odd), then \(P(G_{-A}, N_{G}(W), B)\) is odd (even) or infinite. There is at least one vertex subset $W \subseteq A$, such that \(P(G_{-A}, N_{G}(W), B)\) is not infinite and hence odd (even), namely $W = \{a\}$ where $a \in A$ is connected by an $A$-$B$-path to a vertex $b \in B$.
\end{lemm}

\begin{proof}
We show the first part by contradiction. Assume \(P(G, A, B)\) is even (odd) and there is for a subset \(W \subseteq A\) an even (odd) \(N_{G}(W)\)-\(B\)-path in \(G_{-A}\), connecting a vertex \(x \in N_{G}(W)\) with a vertex \(b \in B\). Because \(x \in N_{G}(W)\), there is a vertex \(a \in W \subseteq A\), such that \(a\) and \(x\) are adjacent and hence the path from \(a\) to \(x\) to \(b\) in \(G\) must have odd (even) length, which contradicts the assumption of the statement. 

As \(P(G, A, B)\) is even (odd), there is at least one \(A\)-\(B\)-path \(P\) in \(G\). Thus, there is a vertex $a \in A$ connected by an $A$-$B$-path to a vertex $b \in B$. Consequently there is a \(N_{G}(a)\)-\(B\)-path in \(G_{-A}\), which proves the second part.
\end{proof}

\begin{theo}
\label{theo:gmsc}
\label{theo:gmsc1}
Let \(G = (V, E)\) be a graph and \(A, B \subseteq V\) two vertex subsets of \(G\). Then
\begin{align}
\Delta(G, A, B) &= \sigma(G_{-A}) \cdot \sigma(G_{-B}) - \sigma(G) \cdot \sigma(G_{-A-B}) \notag \\
& \begin{cases}
< 0 & \text{ if } P(G, A, B) \text{ is even,} \\
= 0 & \text{ if } P(G, A, B) \text{ is infinite,} \\
> 0 & \text{ if } P(G, A, B) \text{ is odd.}
\end{cases}
\end{align}
\end{theo}

\begin{proof}
If \(P(G, A, B)\) is infinite, then \(\Delta(G, A, B) = 0\) by Corollary \ref{coro:delta_infty}. Hence we can assume that \(P(G, A, B)\) is even or odd, that means there is at least one vertex \(a \in A\) connected by an $A$-$B$-path to a vertex $b \in B$. We prove the two cases \(P(G, A, B)\) is even and \(P(G, A, B)\) is odd by induction with respect to the number of vertices in \(G\), denoted by \(n(G)\).

For the basic step we assume a graph \(G\) with the minimal number of vertices, this is \(n(G) = 1\) if \(P(G, A, B)\) is even and \(n(G) = 2\) if \(P(G, A, B)\) is odd.

For \(P(G, A, B)\) is even and \(n(G) = 1\) we have \(G = (\{a\}, \emptyset)\) and \(A = B = \{a\}\). Hence
\begin{align*}
\Delta(G, A, B) &= \sigma(G_{-A}) \cdot \sigma(G_{-B}) - \sigma(G) \cdot \sigma(G_{-A-B}) < 0.
\end{align*}

For \(P(G, A, B)\) is odd and \(n(G) =  2\) we have \(G = (\{a,b\}, \{\{a,b\}\})\) and \(A = \{a\}\), \(B = \{b\}\). Hence
\begin{align*}
\Delta(G, A, B) &= \sigma(G_{-A}) \cdot \sigma(G_{-B}) - \sigma(G) \cdot \sigma(G_{-A-B}) > 0.
\end{align*}

We assume as induction hypothesis that the statement holds for any graph with at most \(k\) vertices and consider from now on a graph \(G\) with \(n(G) = k+1\) vertices. 

In case \(P(G, A, B)\) is even, we have to distinguish whether \(A \cap B = C\) is empty or not. If \(P(G, A, B)\) is even and \(C \neq \emptyset\), we have
\begin{align*}
\Delta(G, A, B) &= \sigma(G_{-A}) \cdot \sigma(G_{-B}) - \sigma(G) \cdot \sigma(G_{-A-B}) \\
&= \sigma(G_{-A}) \cdot \sigma(G_{-B}) - \sum_{\mathclap{\substack{W \subseteq C \\ W \text{ is independent}}}}{\sigma(G_{-C-N_{G}(W)})} \cdot \sigma(G_{-A-B}) \\
&= \sigma(G_{-A}) \cdot \sigma(G_{-B}) - \sigma(G_{-C}) \cdot \sigma(G_{-A-B}) \\
& \eqspace - \sum_{\mathclap{\substack{\emptyset \subset W \subseteq C \\ W \text{ is independent}}}}{\sigma(G_{-C-N_{G}(W)})} \cdot \sigma(G_{-A-B}) \\
& < \sigma(G_{-A}) \cdot \sigma(G_{-B}) - \sigma(G_{-C}) \cdot \sigma(G_{-A-B}) \\
&= \sigma(G_{-C-(A \setminus C)}) \cdot \sigma(G_{-C-(B \setminus C)}) - \sigma(G_{-C}) \cdot \sigma(G_{-C-(A \setminus C) - (B \setminus C)}) \\
&= \Delta(G_{-C}, A \setminus C, B \setminus C).
\end{align*}
Because $P(G_{-C}, A \setminus C, B \setminus C)$ is either even or infinite and the number of vertices in \(G_{-C}\) is at most \(k\), we can apply the induction hypothesis and thus
\begin{align*}
\Delta(G, A, B) &< \Delta(G_{-C}, A \setminus C, B \setminus C) \leq 0.
\end{align*}

Otherwise, if \(P(G, A, B)\) is even and \(C = \emptyset\), we have
\begin{align*}
\Delta(G, A, B) &= \sigma(G_{-A}) \cdot \sigma(G_{-B}) - \sigma(G) \cdot \sigma(G_{-A-B}) \\
&= \sigma(G_{-A}) \cdot \sum_{\mathclap{\substack{W \subseteq A \\ W \text{ is independent}}}}{\sigma(G_{-B-A-N_{G_{-B}}(W)})} - \sum_{\mathclap{\substack{W \subseteq A \\ W \text{ is independent}}}}{\sigma(G_{-A-N_{G}(W)})} \cdot \sigma(G_{-A-B}) \\
&= \phantom{-} \sum_{\mathclap{\substack{W \subseteq A \\ W \text{ is independent}}}}{\left[ \sigma(G_{-A}) \cdot \sigma(G_{-B-A-N_{G_{-B}}(W)}) - \sigma(G_{-A-N_{G}(W)}) \cdot \sigma(G_{-A-B}) \right]}.
\end{align*}
No vertex of \(A\) is adjacent to a vertex of \(B\), because \(P(G, A, B)\) is even. Consequently, the sets \(N_{G_{-B}}(W)\) and \(N_{G}(W)\) coincide for all \(W \subseteq A\) and hence
\begin{align*}
\Delta(G, A, B) &= \phantom{-} \sum_{\mathclap{\substack{W \subseteq A \\ W \text{ is independent}}}}{\left[\sigma(G_{-A}) \cdot \sigma(G_{-A-B-N_{G}(W)}) - \sigma(G_{-A-N_{G}(W)}) \cdot \sigma(G_{-A-B})\right]} \\
&= - \sum_{\mathclap{\substack{W \subseteq A \\ W \text{ is independent}}}}{\left[ \sigma(G_{-A-N_{G}(W)}) \cdot \sigma(G_{-A-B}) - \sigma(G_{-A}) \cdot \sigma(G_{-A-N_{G}(W)-B}) \right]} \\
&= - \sum_{\mathclap{\substack{W \subseteq A \\ W \text{ is independent}}}}{\Delta(G_{-A}, N_{G}(W), B)}.
\end{align*}
For each vertex subset \(W \subseteq A\), the number of vertices in \(G_{-A}\) is at most \(k\), hence we can use the induction hypothesis. By Lemma \ref{lemm:gmsc_prop} and because \(P(G, A, B)\) is even, for each \(W \subseteq A\) we have \(P(G_{-A}, N_{G}(W), B)\) is either infinite and 
\begin{align*}
\Delta(G_{-A}, N_{G}(W), B) &= 0,
\end{align*}
or \(P(G_{-A}, N_{G}(W), B)\) is odd and 
\begin{align*}
\Delta(G_{-A}, N_{G}(W), B) &> 0,
\end{align*}
where the last occurs at least once for \(W = \{a\}\). Thus
\begin{align*}
\Delta(G, A, B) &< 0.
\end{align*}

In case \(P(G, A, B)\) is odd, we have
\begin{align*}
\Delta(G, A, B) &= \sigma(G_{-A}) \cdot \sigma(G_{-B}) - \sigma(G) \cdot \sigma(G_{-A-B}) \\
&= \sigma(G_{-A}) \cdot \sum_{\mathclap{\substack{W \subseteq A \\ W \text{ is independent}}}}{\sigma(G_{-B-A-N_{G_{-B}}(W)})} - \sum_{\mathclap{\substack{W \subseteq A \\ W \text{ is independent}}}}{\sigma(G_{-A-N_{G}(W)})} \cdot \sigma(G_{-A-B}) \\
&= \phantom{-} \sum_{\mathclap{\substack{W \subseteq A \\ W \text{ is independent}}}}{\left[ \sigma(G_{-A}) \cdot \sigma(G_{-B-A-N_{G_{-B}}(W)}) - \sigma(G_{-A-N_{G}(W)}) \cdot \sigma(G_{-A-B}) \right]}.
\end{align*}
The relation \(N_{G_{-B}}(W) \subseteq N_{G}(W)\) implies (as a consequence of Equation \eqref{eq:sigma_prop_1}) that \(\sigma(G_{-B-A-N_{G_{-B}}(W)}) \geq \sigma(G_{-A-B-N_{G}(W)})\), hence
\begin{align*}
\Delta(G, A, B)
& \geq \phantom{-} \sum_{\mathclap{\substack{W \subseteq A \\ W \text{ is independent}}}}{\left[ \sigma(G_{-A}) \cdot \sigma(G_{-A-B-N_{G}(W)}) - \sigma(G_{-A-N_{G}(W)}) \cdot \sigma(G_{-A-B}) \right]} \\
&= - \sum_{\mathclap{\substack{W \subseteq A \\ W \text{ is independent}}}}{\left[ \sigma(G_{-A-N_{G}(W)}) \cdot \sigma(G_{-A-B}) - \sigma(G_{-A}) \cdot \sigma(G_{-A-B-N_{G}(W)}) \right]} \\
&= - \sum_{\mathclap{\substack{W \subseteq A \\ W \text{ is independent}}}}{\Delta(G_{-A}, N_{G}(W), B)}.
\end{align*}
For each vertex subset \(W \subseteq A\), the number of vertices in \(G_{-A}\) is at most \(k\), hence we can use the induction hypothesis. By Lemma \ref{lemm:gmsc_prop} and because \(P(G, A, B)\) is odd, for each \(W \subseteq A\) we have \(P(G_{-A}, N_{G}(W), B)\) is either infinite and 
\begin{align*}
\Delta(G_{-A}, N_{G}(W), B) &= 0,
\end{align*}
or \(P(G_{-A}, N_{G}(W), B)\) is even and 
\begin{align*}
\Delta(G_{-A}, N_{G}(W), B) &< 0,
\end{align*}
where the last occurs at least once for \(W = \{a\}\). Thus
\begin{align*}
\Delta(G, A, B) &> 0. \qedhere
\end{align*}
\end{proof}

If two vertices have even (odd) distance in a bipartite graph, then all paths have even (odd) length and hence the previous theorem proves the Merrifield-Simmons conjecture (Conjecture \ref{conj:msc}) for bipartite graphs:
\begin{coro}
\label{coro:msc_biparite_graph}
Let \(G=(V, E)\) be a bipartite graph and \(u, v \in V\) two vertices of \(G\). Then
\begin{align}
\Delta(G, u, v) &= \sigma(G_{-u}) \cdot \sigma(G_{-v}) - \sigma(G) \cdot \sigma(G_{-u-v}) \notag \\
&\begin{cases}
< 0 & \text{if } d_G(u, v) \text{ is even,} \\
= 0 & \text{if } d_G(u, v) \text{ is infinite,} \\
> 0 & \text{if } d_G(u, v) \text{ is odd.}
\end{cases}
\end{align}
\end{coro}

Consequently, Corollary \ref{coro:msc_biparite_graph} remains valid for graphs obtained from bipartite graphs by attaching arbitrary additional graphs, which are connected to exactly one vertex of the bipartite graph.

%force UTF8 by äöü

%% file: paper.bbl
\begin{thebibliography}{1}

\bibitem{diestel2005}
Reinhard Diestel.
\newblock {\em Graph Theory}, volume 173 of {\em Graduate Texts in
  Mathematics}.
\newblock Springer, Berlin/Heidelberg/New York, 3rd edition, 2005.
\newblock Available from: \url{http://diestel-graph-theory.com}.

\bibitem{gutman1990}
Ivan Gutman.
\newblock Graph propagators.
\newblock {\em Graph Theory Notes of New York}, XIX(3):26 -- 30, 1990.

\bibitem{gutman1991}
Ivan Gutman.
\newblock An identity for the independence polynomials of trees.
\newblock {\em Publications de L'Institute Mathématique}, 50(64):19 -- 23,
  1991.
\newblock Available from:
  \url{http://www.emis.de/journals/PIMB/064/n064p019.pdf}.

\bibitem{li1996}
Xueliang Li.
\newblock On a conjecture of {M}errifield and {S}immons.
\newblock {\em Australasian Journal of Combinatorics}, 14(1):15 -- 20,
  September 1996.
\newblock Available from:
  \url{http://ajc.maths.uq.edu.au/pdf/14/ajc-v14-p15.pdf}.

\bibitem{merrifield1981}
Richard~E. Merrifield and Howard~E. Simmons.
\newblock Enumeration of structure-sensitive graphical subsets: {T}heory.
\newblock {\em Proceedings of the National Acadamy of Sciences of the USA},
  78(2):692 -- 695, February 1981.
\newblock Available from:
  \url{http://www.pnas.org/content/78/2/692.full.pdf+html}.

\bibitem{merrifield1989}
Richard~E. Merrifield and Howard~E. Simmons.
\newblock {\em Topological Methods in Chemistry}.
\newblock Wiley, New York, 1989.

\bibitem{wang2001}
Yong Wang, Xueliang Li, and Ivan Gutman.
\newblock More examples and counterexamples for a conjecture of {M}errifield
  and {S}immons.
\newblock {\em Publications de l'Institut Mathématique}, 69(83):41 -- 50,
  2001.
\newblock Available from:
  \url{http://www.emis.ams.org/journals/PIMB/083/n083p041.pdf}.

\end{thebibliography}
